\documentclass[12pt, reqno]{amsart}

\usepackage{amsmath, amsthm, amscd, amsfonts, amssymb, graphicx, color}
\usepackage[bookmarksnumbered, colorlinks, plainpages]{hyperref}

   \makeatletter \oddsidemargin.9375in
\evensidemargin \oddsidemargin \marginparwidth 2in \makeatother

\newtheorem{theorem}{Theorem}[section]
\newtheorem{cor}[theorem]{Corollary}
\newtheorem{lem}[theorem]{Lemma}
\newtheorem{prop}[theorem]{Proposition}

\numberwithin{equation}{section}

\newcommand{\BN}{{\Bbb  N}}
\newcommand{\BR}{{\Bbb  R}}

\newcommand{\Lip}{{\rm Lip}}
\newcommand{\lip}{{\rm lip}}

\begin{document}
%
\title[Lattice isomorphisms between certain sublattices of continuous functions]{Lattice isomorphisms between certain sublattices of continuous functions}

\author{Vahid Ehsani and  Fereshteh Sady$^1$}



\maketitle
\begin{center}

\address{{\em   Department of Pure
Mathematics, Faculty of \\ Mathematical Sciences,
 Tarbiat Modares University,\\ Tehran, 14115-134, Iran}}

\vspace*{.25cm}
  \email{v.ehsani@modares.ac.ir, sady@modares.ac.ir}
\end{center}
\footnote{$^1$ Corresponding author}

\maketitle
%
\begin{abstract}
Let $C(X,I)$ be the lattice of all continuous functions on a
compact Hausdorff space $X$ with values in the unit interval
$I=[0,1]$.  We show that for compact Hausdorff
spaces $X$ and $Y$ and (not necessarily contain constants) sublattices $A$ and $B$ of $C(X,I)$
and $C(Y,I)$, respectively,  which satisfy a certain separation property,  any lattice isomorphism $\varphi : A
\longrightarrow B$ induces a homeomorphism $\mu: Y \longrightarrow X$.
If, furthermore,  $A$ and $B$ are closed under the multiplication, then $\varphi$ has a representation
$\varphi(f)(y)=m_y(f(\mu(y)))$, $f\in A$, for all points $y$ in a dense $G_\delta$ subset $Y_0$ of $Y$, where each $m_y$ is a strictly increasing continuous bijection on $I$. In particular, for the case where $X$ and $Y$ are metric spaces and $A$ and $B$ are the lattices of all Lipschitz functions with values in $I$, the set $Y_0$ is the whole of $Y$.

\end{abstract}

\section{introduction and Preliminaries}
The problem of determining how the different structures of various spaces (algebras) of functions interacts with each other is an important and active area of research.  Some related problems are to investigate the automatic continuity and also the general form of mappings between these spaces preserving some algebraic or topological properties. For instance, by Kaplansky's theorem \cite{Kap1},  the lattice structure of $C_\BR(X)$, the lattice of all real-valued continuous functions on a compact Hausdorff space $X$,  is determined by the topological structure of $X$. More precisely,  if the lattices $C_\BR(X)$ and $C_\BR(Y)$ are isomorphic, then $X$ and $Y$ are homeomorphic.  Moreover, by \cite{Kap2}, if $X$ is metrizable, then such a map $T:C_\BR(X) \longrightarrow C_\BR(Y)$ is continuous with respect to the supremum  norm and has the general form
\[Tf(y) = t(y, f(\tau(y))) \qquad \qquad (f\in C(X), y\in Y),\]
where $\tau: Y\longrightarrow X$ is a homeomorphism and $t : Y\times \BR \longrightarrow \BR$  is a continuous function.
The case that $T: C_\BR(X) \longrightarrow C_\BR(Y)$ is a lattice homomorphism has been studied in \cite{San1}. In this case
$T$ induces a continuous mapping $\tau: Y_T \longrightarrow  X$ , where $Y_T$ is an open subset of $Y$
such that $t(y, f(\tau(y))^-)\le Tf(y) \le t(y, f(\tau(y))^+) $ for all $f\in C(X )$ and $y\in Y_T$.  Moreover, if $T$ is bijective, then
$\tau$ is a homeomorphism between $Y$ and $X$,  and, furthermore,  $T$ has a representation $Tf(y)= t(y, f(\tau(y)))$ for all $f\in C_\BR(X)$ and all $y$ in a dense $G_\delta$ subset of $Y$.

It is easy to see that a bijection $T:C_\BR(X) \longrightarrow C_\BR(Y)$ is a lattice isomorphism if and only if $T$ is order preserving in both directions, that is $f\le g$ if and only if $Tf\le Tg$ for all $f,g\in C_\BR(X)$. However, linear order preserving bijections $T: A(X) \longrightarrow A(Y)$ between subspaces $A(X)$ and $A(Y)$ of $C_\BR(X)$ and $C_\BR(Y)$ have been studied in \cite{Le-Li}. It was shown that under some conditions on the subspaces, such a map is a weighted composition operator inducing a homeomorphism between $X$ and $Y$. Order preserving bijections on the set $C_+(X)$ of all continuous functions on a compact space $X$ with values in $[0,\infty)$  have also been considered in \cite{Mar0}.

Motivated by the Moln\'ar's result \cite{Mol} on sequential
isomorphisms between the sets of von Neumann algebra effects,  Marovt considered  the order preserving bijections and multiplicative bijections on the set  $C(X,I)$ of continuous functions on a compact Hausdorff space $X$ with values in the unit interval $I=[0,1]$ in \cite{Mar2,Mar1}.
Clearly $C(X,I)$ is a  lattice with the usual ordering of functions and a semigroup with pointwise multiplication.  It should be noted that multiplicative bijections on $C(X,I)$ are order preserving in both directions, and so they are lattice isomorphisms, see \cite[Lemma 3]{San2}. By \cite{Mar2} if $X$ is first countable, and $T: C(X,I)\longrightarrow C(X,I)$ is a multiplicative bijection, then  there exist a homeomorphism $\tau$ on $X$ and a continuous map
$k : X \longrightarrow (0,\infty)$, such that $Tf(x) = f(\tau(x))^{k(x)}$ for all $x\in X$ and all
$f\in C(X, I)$. The same problem for arbitrary compact Hausdorff spaces $X$ has been studied in \cite{San2}.
By \cite{Mar1},  if $X$ is first countable, then  for any  bijection $ \varphi : C(X,I) \longrightarrow C(X,I)  $  which preserve the order in both directions there exist a homeomorphism $\mu: X
\longrightarrow X$ and a family $\{m_x\}_{x\in X}$ of increasing
continuous bijections on $I$ such that
\[ \varphi(f)(x)=m_x(f(\mu(x))) \qquad \qquad (f\in C(X,I), x\in X).\]
Marovt conjectured that the results in \cite{Mar2,Mar1} are also valid  without assuming
that $X$ is first countable. The conjecture for multiplicative (and consequently for order preserving) bijections on $C(X,I)$ was disproved by
Ercan and \"{O}nal in \cite{Er-On}. The more general problem of
characterizing those compact Hausdorff spaces $X$ for which any
multiplicative bijection $T: C(X,I) \longrightarrow C(X,I)$ has
the above standard form has been investigated in \cite{Ar}.

In this paper, we assume that $X$ and $Y$ are compact Hausdorff spaces and we study lattice isomorphisms $\varphi:A \longrightarrow B$ between sublattices $A$ and $B$ of $C(X,I)$ and $C(Y,I)$, respectively, not necessarily contain the constants. We show that under a mild separation property on $A$ and $B$, $\varphi$ induces a homeomorphism $\mu: Y\longrightarrow X$. Then, under the assumption that $A$ and $B$ are closed under multiplication we give a description of $\varphi$ on a dense $G_\delta$ subset of $Y$ and investigate continuity of $\varphi$ with respect to the uniform convergence topology.

Let us fix some notations. For a compact Hausdorff space  $X$,  and $f,g \in C(X,I)$,  the notations $f \vee g$ and  $f \wedge g $ stand for $\max(f,g)$ and  $\min(f,g)$, respectively. For compact Hausdorff spaces $X$ and $Y$, a map  $ \varphi: A \longrightarrow B $ between sublattices $A$ and $B$ of $C(X,I)$ and $C(Y,I)$, respectively, is a lattice homomorphism  if $\varphi (f \wedge g) = \varphi(f) \wedge \varphi(g)$ and $\varphi(f \vee g) = \varphi(f) \vee \varphi(g)$ for all $f,g\in A$. A bijective homomorphism is  called a lattice isomorphism.

For $f\in C(X,I)$, $z(f)$ denotes the zero set of $f$ and ${\rm coz}(f)$ is its cozero set, that is ${\rm coz}(f)=X\backslash z(f)$.

For a compact Hausdorff space $X$, we say that a subset $A$ of $C(X,I)$ has {\em Urysohn's property}, if for any pair of disjoint closed subsets $F$ and $G$ of $X$ there exists $f\in A$ such that $f=0$ on $F$ and $f=1$ on $G$ (compare with the Property 1 in \cite{Ar}).  We also say that {\em the evaluation of $A$ on $X$ is dense in} (respectively, equal to) $I$ if for each $x\in X$, the set $A_x=\{f(x) : f \in A\}$ is dense in (respectively, equal to) $I$. It is obvious that if $A$ contains the constant functions, then the evaluation of  $A$ on $X$ is equal to $I$.

For a compact metric space  $(X,d)$, the sublattice $\Lip(X,I)$ of $C(X,I)$ consisting of all $I$-valued Lipschitz functions on $X$ has Urysohn's property. Indeed, for disjoint closed subsets $F$ and $G$ of $X$, the function $f\in \Lip(X,I)$ defined by $f(x)=\min(\frac{d(x,F)}{d(F,G)},1)$ satisfies the desired property. More generally, for each $\alpha\in (0,1]$, the set $\Lip_\alpha(X,I)$ of all functions $f: X \longrightarrow I$ satisfying the Lipschitz condition of order $\alpha$ is a sublattice of $C(X,I)$ which has Urysohn's property. Similarly for $\alpha\in (0,1)$, the sublattice $\lip_\alpha(X,I)$ of $\Lip_\alpha(X,I)$ consisting of all functions $f\in \Lip_\alpha(X,I)$ with $\lim \frac{f(x)-f(y)}{d^\alpha(x,y)}=0$ as $d(x,y)\to 0$, also has this property.
For another example of such sublattices we can refer to the lattice ${\rm AC}(X,I)$ of all $I$-valued absolutely continuous functions on a compact subset $X$ of the real line.  It should be noted that the above mentioned  sublattices of $C(X,I)$  are all closed under multiplication.

\section{Main Results}
Our main result is as follows.

\begin{theorem}\label{t2}
 Let $X$ and $Y$ be compact Hausdorff spaces, $A$ and $B$ be sublattices of $C(X,I)$ and $C(Y,I)$, respectively having Urysohn's
 property and  $\varphi:A \longrightarrow B$ be a lattice isomorphism. Then $\varphi$ induces a homeomorphism $\mu$ from $Y$ onto $X$. If, furthermore,

 {\rm (i)} $A$ and $B$ are closed under multiplication, and

 {\rm (ii)}  the evaluations of $A$ and $B$ are dense in $I$,

 \noindent
then there exists a dense $G_\delta$ subset $Y_0$ of $Y$ and a family  $\{m_y\}_{y\in Y_0}$ of  strictly increasing  continuous bijection on $I$ such that
 \[\varphi(f)(y) = m_{y}( f(\mu(y)) )\, \qquad \qquad (f \in A, y\in Y_0).\]
 \end{theorem}
 We prove the theorem through the following lemmas.

In what follows we assume that $X$ and $Y$ are compact Hausdorff
spaces, $A$ and $B$ are sublattices of $C(X,I)$ and $C(Y,I)$,
respectively which have Urysohn's property, and $\varphi:
A\longrightarrow B$ is a lattice isomorphism. Then
$\varphi$ is order preserving in both
directions, that is  $f\le g$ if and only if $\varphi(f) \le \varphi(g)$ for
all $f,g\in A$.

Since $X$ is compact, it follows from the hypotheses that $A$ contains
the constant functions $0$ and $1$ and, furthermore, $\varphi(0)=0$ and
$\varphi(1)=1$.

We note that for any pair of functions $f,g\in A$ we have $fg=0$
if and only if $f \wedge g=0$. Hence we get the next lemma.
\begin{lem}\label{l1}
For $f,g\in A$, we have $fg=0$ if and only if
$\varphi(f)\varphi(g)=0$.
\end{lem}
Using  Urysohn's property, for any pair  $U_1$ and $U_2$
of open subsets of $X$ with disjoint closures we can find $f_1, f_2\in A$ such
that $f_1=0$ on $U_1$, $f_2=0$ on $U_{2}$ and $f_1 \vee
f_2=1$.

For each point $x \in X$, let $\mathcal{U}_{x}$ denote
the set of all functions $f\in A$ vanishing on a neighborhood of
$x$. We  put
\begin{equation*}
C_{x}=\lbrace y\in Y: \varphi(f)(y)=0\; {\rm  for\; all}\;f \in
\mathcal{U}_{x}\rbrace .
\end{equation*}
Similar notations $\mathcal{U}_{y}$ and $C_{y}$ will be used for
each $y\in Y$.
\begin{lem}\label{l2}
For each $x\in X$, $C_x$ is nonempty.
\end{lem}
\begin{proof}
First we note that $ C_{x}= \bigcap_{f\in \mathcal{U}_{x} }
\varphi (f)^{-1} (\{0 \})$. Hence $C_x$  is a closed subset of
$Y$. Using finite intersection property, it suffices to show that
for each $f_1,..., f_n\in \mathcal{U}_x$, the intersection
$\cap_{i=1}^{n} \varphi (f_i)^{-1} (\{0 \})$ is nonempty.  For
this, assume that $ f_{1}, f_{2}, \cdots , f_{n} \in
\mathcal{U}_{x}$. Since $\varphi$ is a lattice isomorphism, we have $ \bigcap _{i=1} ^{n}\varphi (f_{i})^{-1}
(\lbrace 0 \rbrace)= \lbrace y\in Y: \varphi (\bigvee_{i=1}
^{n}f_{i})(y)=0 \rbrace $. Put $f_{0}= \bigvee_{i=1} ^{n}f_{i} $
and assume on the contrary that $\varphi (f_0)(y) \neq 0$ for all
$y\in Y$. Then clearly $ f_{0}  \neq 0$, that is $ z(f_{0})=
\lbrace x\in X: f_{0}(x)=0 \rbrace \neq X $. Since for each  $i\in
\{1, 2,..., n \}$, $ f_{i}=0$ on a neighborhood $U_i$ of $x$, it
follows that $f_0=0$ on the neighborhood $U=\cap_{i=1}^{n} U_i$
of $x$, i.e. $U\subseteq z(f_0)$. In particular, $\overline{U}\neq
X$.  Using  Urysohn's property, we can easily find a nonzero function
$ g\in A $ such that $ {\rm supp}(g) \subseteq U$. Then $ f_{0}
g=0$, and so $ \varphi(f_{0}) \varphi(g)= 0 $. Since
$z(\varphi(f_0))=\emptyset$ we get $z(\varphi (g))= Y $. Thus $
\varphi(g)=0$ which is impossible,  since $\varphi(0)=0$
and $\varphi$ is a bijection. Hence $C_{x} \neq \emptyset $.
\end{proof}
\begin{lem}\label{l3}
For each $x\in X$ the  inclusion $\varphi
(\mathcal{U}_{x})\subseteq \mathcal{U}_{y}$ holds for all $y\in
C_{x}$.
\end{lem}
\begin{proof}
Let $y\in C_x$ and let $f\in \mathcal{U} _{x}$ be nonzero. Then there exists a neighborhood $U$
of $ x $ such that $ f=0 $ on $U$.  Choose open neighborhoods
$U_{1}$ and $U_{2}$ of $x$ such that $\overline{U_{2}} \subseteq
U_{1}\subseteq \overline{U_{1}} \subseteq U$. Then there are
functions $ f_{1}, f_{2} \in A $ such that $ f_{1}=1$ on $ U_{1} $
and $ {\rm supp} (f_{1})\subseteq U $ and similarly $ f_{2} = 0$
on $\overline{U_{2}} $ and  $ f_{2}=1 $ on $ X \backslash U_{1} $.
Hence  $ f_{1} \vee f_{2}= 1$ and consequently $ \varphi
(f_{1}) \vee \varphi (f_{2})= 1$. We note that $ \varphi
(f_{2})(y)=0 $, since $ y\in C_{x} $. Thus $ \varphi (f_{1})(y)=1
$. On the other hand $ f_{1}f=0 $ and it follows from Lemma
\ref{l1} that $\varphi(f_1) \varphi(f)=0$. Hence the open neighborhood  ${\rm coz}(\varphi(f_{1}))$
of $y$ is contained in  $z(\varphi(f)) $, as desired.
\end{proof}

\begin{lem}\label{l4}
For each $x\in X$ the set $C_x$ is a singleton. A similar
assertion holds for all $y\in Y$.
\end{lem}
\begin{proof}
 We, first, show that for each $y\in C_x$ we have $C_y=\lbrace x \rbrace$. Let $y\in C_x$. Using Lemma \ref{l2} for $\varphi^{-1}$ instead of $\varphi$, we conclude that $C_y$ is nonempty.
  Hence we need only to show that for any point $z\in X$ distinct from $x$ we have  $z\notin C_y$. Assume on the contrary that $z\neq x$ and $z\in
  C_{y}$. Since $y\in C_{x}$, it follows from the above lemma  that
  $\varphi(\mathcal{U}_{x})\subseteq \mathcal{U} _{y}$. Using  the same lemma for $\varphi^{-1}$,  we have
  $\varphi^{-1} (\mathcal{U}_{y})\subseteq \mathcal{U}_{z}$. Hence $ \mathcal{U}_{x} \subseteq \mathcal{U}_{z}$
  which concludes, by  Urysohn's property, that $x=z$, a contradiction.

Now for an arbitrary point $y\in Y$, since $C_{y}$ is nonempty, we
can choose a point $z\in C_{y}$. Then similar to the above
argument we have  $C_{z}=\{y\}$, that is $y\in C_z$. Hence, using
the above argument once again we have $C_y=\{z\}$, i.e. $C_y$ is a
singleton.

Similarly for each $x\in X$, $C_x$ is a singleton.
\end{proof}
By the above lemma  we can define a map $\mu: Y \longrightarrow
X$ which associates to each $y\in Y$ the unique point $x\in C_y$.
\begin{lem}\label{l5}
The map $\mu:Y \longrightarrow X$ is a homeomorphism.
\end{lem}
\begin{proof}
As it was noted in the proof of Lemma \ref{l4}, for each $y\in Y$ and $x\in X$, we have $C_x=\{y\}$ if and only if $C_y=\{x\}$. This implies that  $\mu$ is a bijective map.
To show that $\mu$ is continuous, assume on the contrary that there exists a net $\{y_\alpha\}$ in 
$Y$ converging to a point $y_0\in Y$ and $\{\mu(y_\alpha)\}$ does not converge to $\mu(y_0)$. Then passing through a subnet we may assume that there exists an open neighbourhood $U$ of $\mu(y_0)$ such that $\mu(y_\alpha)\in X \backslash \overline{U}$ for all $\alpha$. Choose open neighborhoods $U_0$, $U_1$ and $U_2$ of $\mu(y_0)$ such that $\overline{U_i}\subseteq U_{i+1}$ for $i=0,1$ and $\overline{U_2}\subseteq U$. Then there are functions $f_1,f_2\in A$ such that $f_1=1$ on $U_1$ and $f_1=0$ on $X\backslash U_2$, and similarly $f_2=0$ on $U_0$ and $f_2=1$ on $X\backslash U_1$.  Then clearly $f_1\vee f_2=1$ and consequently $\varphi(f_1) \vee \varphi(f_2)=1$. Since $f_2=0$ on the neighborhood $U_0$ of $\mu(y_0)$ it follows that $\varphi(f_2)(y_0)=0$. Hence $\varphi(f_1)(y_0)=1$. On the other hand $f_1=0$ on the open subset $X\backslash \overline{U}$ of $X$ (containing all $\mu(y_\alpha)$) which implies that $\varphi(f_1)(y_\alpha)=0$ for all $\alpha$, a contradiction. 
\end{proof}
The above lemma completes the proof of the first part of Theorem \ref{t2}.

We should note that since $\varphi^{-1}$ is also a lattice isomorphism it induces a homeomorphism $\nu: X \longrightarrow Y$. For a given $y\in Y$ if we put $x=\mu(y)$ and $z=\nu(x)$, then, using Lemma \ref{l3} we have $\varphi(\mathcal{U}_x)\subseteq \mathcal{U}_y$ and also
$\varphi^{-1}(\mathcal{U}_z)\subseteq \mathcal{U}_x$. Hence $\mathcal{U}_z\subseteq \mathcal{U}_y$, which implies that $y=z$. That is $\nu(\mu(y))=y$. Similarly we have $\mu(\nu(x))=x$. Thus $\mu^{-1}$ is, indeed, the associated homeomorphism to $\varphi^{-1}$.

 Now  we investigate  more properties of $\mu$ and its relations to
$\varphi$ whenever  $A$ and $B$ are, in addition, closed under multiplication. First we extend the property stated in Lemma \ref{l3}.

\begin{lem}\label{c1}
Assume, furthermore, that $A$ and $B$ are  closed under multiplication. If  $f,g \in A$ and $y\in Y$
such that $f=g$ on a neighborhood of $\mu(y)$, then $\varphi(f) =
\varphi(g)$ on a neighborhood of $y$.
\end{lem}
\begin{proof}
Let $U$ be an open neighborhood of $\mu(y)$ in
$ X $ such that $f= g$ on $U$. Choose an open
neighborhood $U_{1}$ of $\mu(y)$ with  $\overline{U_{1}} \subseteq
U$. Using  Urysohn's property of $A$, there are functions
$f_{1}, f_{2} \in A $ such that $ f_{1}=0$ on $X\backslash U$ ,$f_{2}=0 $ on $ U_{1}
$ and $ f_{1} \vee f_{2} =1 $.  Hence, $ff_1$ and $ff_2$ are elements of $A$ such that $f=ff_1 \vee ff_2$. Since
$f=g$ on $U$, we conclude that $ ff_{1} \leq g $ on $ U $. On the
other hand $ ff_{1} =0 $ on $ X\backslash U $. Thus we have $ff_{1} \leq g $ and consequently $\varphi(ff_1)\le \varphi(g)$. Since $ff_{2}=0$ on $U_1$, it follows from Lemma \ref{l3} that $\varphi(ff_{2})=0$ on  an  open
neighborhood $V$ of $y$. Thus  $ \varphi(f)=\varphi(ff_1) \vee \varphi(ff_2)=\varphi(ff_{1})$ on $V$, which concludes that $\varphi(f) \leq \varphi(g)$ on $V$.

A similar argument shows that $\varphi(f)\geq \varphi(g) $ on an open neighborhood of $y$, as desired.
\end{proof}

\begin{lem}\label{l6}
Assume that $A$ and $B$ are closed under multiplication. For each $f,g\in A$ and $y\in Y$, the inequality $f(\mu(y))<g(\mu(y))$ implies that $\varphi(f)(y) \le \varphi(g)(y)$.
\end{lem}

\begin{proof}
By hypotheses $f<g$ on a neighborhood of $U$ of $\mu(y)$ in $X$, in particular,  $f=f\wedge g$ on $U$.
Hence, by the above lemma we have  $\varphi(f)=\varphi(f\wedge  g)$ on a neighborhood of $y$. Thus $\varphi(f)\le \varphi(g)$ on this neighborhood.
\end{proof}

We note that since $A$ and $B$ are not assumed to contain the constants we require to state the next lemma.

\begin{lem} \label{l7}
Let $A$ and $B$ be closed under multiplication and the evaluation of $A$ on $X$ be dense in $I$. Then for each interval $(r,s)$ in $I$ there exists a function $ f \in A $, with $r<f(t) <s$ for all $t\in X$.
\end{lem}

\begin{proof}
Let $ x \in X $. Since the evaluation of $A$ on $X$ is dense
in $I$, there is a function $ f_{x} \in A $ such that $ r <
f_{x}(x) < s $. Choose an open neighborhood $U_{x}$ of  $x$ such that $r<f_{x}(t) < s $ for all $t \in
U_{x}$. Let $ W_{x} $ be an open neighborhood of $x$
with $\overline{W_{x}}\subseteq U_x$. By Urysohn's property of $A$
there is a function $g_{x}\in A $ with $g_{x}=1$ on $W_{x}$ and $ {\rm supp}(g_{x})\subseteq U_{x} $. Put $h_{x} = g_{x}f_{x}$.  Then $h_x\in A$ and $ r< h_{x} <s $ on $ W_x $.
Since $ \{W_{x}\}_{x \in X}$
is an open cover of $X$,  there are  $x_{1}, x_{2}, ... , x_{n}
\in X $ such that $X=\cup_{i=1}^{n} W_{x_i} $.  Set  $ f = \bigvee _{i=1}^{n}h_{x_{i}} $.
Then $f\in A$ and since  for any $ t\in X $ there exists  $1\le i\le n$ with  $ t \in W_{x_{i}}$ it follows that $f(t) \ge h_{x_i}(t) > r$. If $f(t)\ge s$ for some $t\in X$, then there exists $1\le j\le n$  such that $  h_{x_j}(t) \ge s$. Therefore, $f_{x_j}(t)\ge h_{x_j}(t)\ge s$ which implies that $t\notin U_{x_j}$. Hence $g_{x_j}(t)=0$ which is impossible, since $s\le h_{x_j}(t)=g_{x_j}(t) f_{x_j}(t)$. This argument shows that for each $t\in X$, we have $r<f(t)<s$, as desired.
 \end{proof}

 For any pair $f,g\in A$ we put
\[I_{f,g}=\{y\in Y: \varphi(f)(y)\neq \varphi(g)(y)\; {\rm  whenever}\;  f(\mu(y)) \neq g(\mu(y))\}\]
and $I_{\varphi}=\cap \{I_{f,g}: f,g\in A\}$

 \begin{lem}\label{p1}
 Under the assumptions of Lemma \ref{l7}, $ I_{\varphi}=Y$ if and only if  $\varphi$ is strictly increasing, that is $f<g$ implies $\varphi(f)<\varphi(g)$ for all $f,g\in A$.
 \end{lem}
 \begin{proof}
Let $\varphi$ be strictly increasing and assume on the contrary that $I_{\varphi} \neq Y$. Then there is $y \in Y$ such that
  $ y \notin I_{\varphi}$. Hence there are functions $ f,g \in A $ such that $f(\mu(y)) < g(\mu(y)) $ and $\varphi(f)(y) = \varphi(g)(y) $.  Choose $r_{0}, r_{1}, s_{0}, s_{1} \in I$ with $ f(\mu(y)) < r_{0} <s_{0}  < r_{1} < s_{1} < g(\mu(y))$.
Let $ U$ be an open neighborhood of $ \mu(y) $ such that $ f(t) < r_{0} <s_{0}  < r_{1} < s_{1} < g(t)  $ for all $t \in U$.
By Lemma \ref{l7} there are functions $ f_{r_{0},s_{0}} , f_{r_{1},s_{1}} \in A$ such that
 $r_0<f_{r_0,s_0}<s_0$ and $r_1<f_{r_1,s_1}<s_1$. Thus, \[f \wedge f_{r_{0},s_{0}}< s_{0} <r_1 < g \vee f_{r_{1},s_{1}}.\]
 Since $\varphi$ is strictly increasing it follows that
 \[\varphi(f \wedge f_{r_{0},s_{0}}) < \varphi(g \vee f_{r_{1},s_{1}}). \]

Clearly $f< f_{r_0,s_0}$ on $U$ and consequently $f<f_{r_0,s_0}
\wedge f$ on $U$. Thus, by Lemma \ref{l6} we have  $\varphi(f)(y)
\le \varphi(f \wedge f_{r_{0},s_{0}})(y)$.  In a similar manner we
have $\varphi(g)(y) \ge \varphi (g \vee  f_{r_1,s_1})(y)$. Hence
$\varphi(f)(y)<\varphi(g)(y)$  which is
a contradiction. This argument shows that $I_\varphi=Y$.

The converse statement is trivial. Indeed, assume that $I_\varphi=Y$ and $y\in Y$. For each $f,g\in A$ with $f<g$, since $y\in I_{f,g}$, we have $\varphi(f)(y)\neq \varphi(g)(y)$, that is $\varphi(f)(y)<\varphi(g)(y)$.
 \end{proof}
 \begin{prop}\label{p2}
Assume that $A$ and $B$ are closed under multiplication and the evaluation of $A$ and $B$ on $X$ and $Y$, respectively are  dense in $I$. If
$\varphi$ is continuous (with respect to the uniform convergence topology), then $ \varphi^{-1} $ is strictly increasing.
\end{prop}

\begin{proof}
Assume on the contrary that $\varphi^{-1}$ is not strictly
increasing. Using Lemma \ref{p1} for $\varphi^{-1}$ instead of
$\varphi$ we have  $I_{\varphi^{-1}} \neq X $. Thus there exists a
point $y_0\in Y$ such that $\mu(y_{0}) \notin I_{\varphi ^{-1}}$.
Therefore,  there are functions $f_0, g_0 \in A$  such that
$f_0(\mu(y_0)) = g_0 (\mu(y_0))$ while  $\varphi(f_0)(y_0) <
\varphi(g_0)(y_{0})$. Replacing $f_0$ by $f_0\vee g_0$ we may
assume that $f_0\le g_0$.  We consider two following cases:

{\bf Case 1}. $0 \leq f_0(\mu(y_0) ) <1$.

In this case we shall show that there exists a sequence $\{f_n\}$
in $A$ converging uniformly to $f$ on $X$ such that
$f_n(\mu(y_0))>f_0(\mu(y_0))$ for all $n\in \Bbb N$.

Given $\epsilon>0$, since $A_{\mu(y_0)}$ is dense in $I$, there exists $h_1\in A$ such that $0<h_1(\mu(y_0))-f_0(\mu(y_0))<\epsilon$. Let $U$ be an open neighbourhood of $\mu(y_0)$ in $X$ such that $0<h_1(t)-f_0(t)<\epsilon$ for all $t\in U$. Choose  an open neighbourhood $V$ of $\mu(y_0)$ such that $\overline{V}\subseteq U$. Then we can find $h_2\in A$ with $h_2=1$ on $V$ and ${\rm supp}(h_2)\subseteq U$. We put $h_\epsilon=h_1h_2$ and $f_\epsilon= h_\epsilon \vee f_0$. Clearly $f_\epsilon\in A$ and $f_\epsilon(\mu(y_0))>f_0(\mu(y_0))$. We claim that $0\le f_\epsilon(t)-f_0(t)\le \epsilon$ for all $t\in X$. Fix a point $t\in  X$. If $t\in U$, then  $h_\epsilon (t)-f_0(t)\le h_1(t)-f_0(t)<\epsilon$. Hence in either of cases that $f_\epsilon(t)=f_0(t)$ or $f_\epsilon(t)=h_\epsilon(t)$ we have  $0\le f_\epsilon(t)-f_0(t)\le \epsilon$.
If $t\in X\backslash U$, then $h_\epsilon(t)=0$ and consequently $f_\epsilon(t)- f_0(t)=f_0(t)-f_0(t)=0$.
This argument shows that $0\le f_\epsilon(t)-f_0(t)\le \epsilon$ for all $t\in X$.
Therefore, there exists a sequence $\{f_n\}_{n\in \Bbb N}$ converging uniformly on $X$ to $f_0$ and for each $n\in \Bbb N$
\[f_n(\mu(y_0))> f_{0}(\mu(y_{0}))=g_0(\mu(y_0)).\]
 Now the continuity assumption on $\varphi$ implies that $\{\varphi(f_{n}) \} $ converges uniformly on $Y$ to $\varphi(f)$.
 On the other hand, since for each $n\in \Bbb N$,
 \[f_n(\mu(y_{0}))  > g_0(\mu(y_{0})) \] it follows from Lemma \ref{l6} that $ \varphi(f_{n})(y_{0}) \geq \varphi(g_0)(y_{0})$  for all $n \in \mathbb{N} $ which is impossible, since $\varphi (f_0)(y_{0})<\varphi(g_0)(y_{0}$.

{\bf Case 2.} $f_{0}(\mu(y_{0})) = 1$.

The argument in this case is a minor modification of Case 1. In
this case we find a sequence $\{f_n\}$ in $A$ converging uniformly
to $f_0$ and such that $f_n=1$ on a negibourhood of $\mu(y_0)$.
For this, given $\epsilon>0$,  we choose an open neighbourhod of
$\mu(y_0)$ such that $1-f_0(t)<\epsilon$ for all $t\in U$ and then
we choose an open neighbourhood $V$ of $\mu(y_0)$ such that
$\overline{V}\subseteq U$.  Then there exists a function
$h_\epsilon\in A$ satisfying $h_\epsilon=1$ on $V$ and ${\rm
supp}(h_\epsilon)\subseteq U$. We put $f_\epsilon= h_\epsilon \vee
f_0$. Then $f_\epsilon\in A$, $f_\epsilon=1$ on $V$ and an easy
verification shows that $|f_\epsilon(t)-f_0(t)|\le \epsilon$ for
all $t\in X$. Hence we can find a sequence $\{f_n\}$ in $A$
satisfying the desired conditions. Since $f_n=1$ on a
neighbourhood of $\mu(y_0)$ it follows from Lemma \ref{c1} that
$\varphi(f_n)(y_0)=1$ for all $n\in \BN$. This  implies that
$\varphi(f_0)(y_0)=1$ which is impossible, since
$\varphi(f_0)(y_0)< \varphi(g_0)(y_0)$.
\end{proof}

In the next Lemmas $A$ and $B$ are, furthermore, closed under multiplication and their evaluations on $X$ and $Y$ are dense in $I$.
\begin{lem}
The set  $I_\varphi$ is a dense  $G_\delta$ subset of $Y$.
\end{lem}
\begin{proof}
 First we show that for each pair $f_0,g_0 \in A$ with $f_0 < g_0$, the set $I_{f_0,g_0}$ is a dense open subset of $Y$. We note that in this case
$I_{f_0,g_0}=\{y\in Y: \varphi(f_0)(y) \neq \varphi(g_0)(y)\}$.
Hence $I_{f_0,g_0}$ is, clearly, an open subset of $Y$. Assume on
the contrary that $Y\backslash I_{f_0,g_0}$ contains an open
subset $V$ of $Y$. Then $\varphi(f_0)=\varphi(g_0)$ on $V$ and
using Lemma \ref{c1} for $\varphi^{-1}$ instead of $\varphi$ we conclude
that $f_0(\mu(y))=g_0(\mu(y))$ for all $y\in V$, which is a
contradiction.

  Now consider the dense subset $J= \cup_{x\in X} A_x$ of $I$. Then clearly  $J$ contains a  countable dense subset $J_0$.
  For each $p,q \in J_0 $ with  $p < q$ choose, by Lemma \ref{l7},  $f_{p,q}\in A$ such that
  $p<f_{p,q}<q$. We shall show that
\[I_\varphi =\cap\{I_{f_{p_1,q_1},f_{p_2,q_2}}: p_1,p_2,q_1,q_2\in
J_0\; {\rm and} \; p_1<q_1<p_2<q_2\}.\] Let $y\in Y\backslash
I_\varphi$. Then there exists $f,g\in A$ such that  $f(\mu(y)) <
g(\mu(y))$ and $\varphi(f)(y) = \varphi(g)(y)$. Choosing  $p_1,
q_1, p_2, q_2 \in J_0$ with $f(\mu(y)) < p_1 < q_1 < p_2 < q_2 <
g(\mu(y))$ we have $f(\mu(y)) < f_{p_1, q_1} < f_{p_2, q_2} <
g(\mu(y))$. Then it follows from Lemma \ref{l6} that
\[\varphi(f)(y) \le \varphi(f_{p_1, q_1})(y) \le \varphi(f_{p_2, q_2})(y) \le \varphi(g)(y).\]
Since $\varphi(f)(y)=\varphi(g)(y)$ we get  $\varphi( f_{p_1,
q_1})(y) = \varphi( f_{p_2, q_2} )(y)$, that is $y\notin
I_{f_{p_1,q_1},f_{p_2,q_2}}$. Hence
\[ \cap\{I_{f_{p_1,q_1},f_{p_2,q_2}}: p_1,p_2,q_1,q_2\in J_0\; {\rm
and} \; p_1<q_1<p_2<q_2\} \subseteq I_\varphi. \] The other
inclusion  is trivial. Thus we get the claimed equality. As  for
each $p_1,p_2,q_1,q_2\in J_0$ with $p_1<q_1<p_2<q_2$,
$I_{f_{p_1,q_1},f_{p_2,q_2}}$ is an open dense subset of $Y$, it
follows from Bair's category theorem that $I_\varphi$ is a dense
$G_\delta$ subset of $Y$.
\end{proof}
 Now we put $Y_0= I_\varphi \cap \mu^{-1}({ I_{\varphi^{-1}}})$.
As in the lemma above we can conclude that $I_{\varphi^{-1}}$ is a dense $G_\delta$ subset
of $X$. Being $\mu^{-1}$ a homeomorphism we again conclude from
Bair's category theorem that $Y_0$ is a dense $G_\delta$ subset of
$Y$. It is obvious that for each $y\in Y_0$, and $f,g\in A$ we
have $f(\mu(y))=g(\mu(y))$ if and only if
$\varphi(f)(y)=\varphi(g)(y)$.

\begin{lem}
 For each $y\in Y_0$, the map $m'_y:A_{\mu(y)}\cap
(0,1) \longrightarrow B_y\cap (0,1)$ defined by
$m'_y(f(\mu(y))=\varphi(f)(y)$, $f\in A$,  is  continuous and
strictly increasing.
\end{lem}
\begin{proof}
Clearly for each $y\in Y_0$, $m'_y$ is well-defined and strictly
increasing. To prove that $m'_y$ is continuous, let $\{r_n\}$ be a
sequence in $A_{\mu(y)}\cap (0,1)$ converging to $r\in
A_{\mu(y)}\cap (0,1)$. We assume  
that $\{r_n\}$ is strictly increasing, since the case that $\{r_n\}$ is strictly decreasing has a similar discussion. Let $f\in A$ such that
$f(\mu(y))=r$ and for each $n\in \BN$, let $f_n\in A$ such that
$f_n(\mu(y))=r_n$. Since $y\in Y_0$ and $r_n<r_{n+1}$ it follows
that $\varphi(f_n)(y)< \varphi(f_{n+1})(y)$ for all $n\in \BN$. If
$\lim \varphi(f_n)(y) < \varphi(f)(y)$, then since $B_y\cap (0,1)$
is dense in $I$, we can find a function $g\in A$ such that $\lim
\varphi(f_n)(y) <\varphi(g)(y)< \varphi(f)(y)$. Hence
$\varphi(f_n)(y)< \varphi(g)(y)< \varphi(f)(y)$ for all $n\in
\BN$. Since $y\in Y_0$ we conclude that $f_n(\mu(y))<g(\mu(y))< f(\mu(y))$
for all $n\in \BN$. Hence $\lim f_n(\mu(y))\le g(\mu(y)) <
f(\mu(y))$, a contradiction. Thus $\lim
\varphi(f_n)(y)=\varphi(f)(y)$, that is $m'_y$ is continuous.
\end{proof}

\begin{lem}
 The map $m'_{y}$ has a unique extension to a continuous bijection on $I$.
\end{lem}
\begin{proof}
We note that for each $s\in (0,1)$ there exists a strictly increasing
sequence $\{s_n\}$ in $A_{\mu(y)}\cap (0,1)$ converging to $s$.
Since $\{m'_y(s_n)\}$ is an increasing sequence in $I$ it follows
that  $\{m'_y(s_n)\}$ converges to a point in $I$. We put
$m_y(s)=\lim m'_y(s_n)$. We should note that the definition of
$m_y(s)$ is independent of the strictly increasing sequence
$\{s_n\}$ converging to $s$.
Indeed assume that  $\{t_n\}$ is also a  strictly increasing sequence in  $A_{\mu(y)}\cap (0,1)$ which converges to $s \in (0,1)$ and assume on the contrary that  $\lim m'_y(s_n) < \lim m'_y(t_n) $ . Hence there exists $N_0 \in \mathbb{N}$ such that  $m'_{y}(s_n) < m'_y(t_{N_0})$ for all $n\in \mathbb{N}$. Since $ m'_y$ is strictly increasing we get $s_n < t_{N_0}$ for all $n \in \mathbb{N} $. Thus $\lim s_{n} \le t_{N_0} < s $ which is a contradiction.

We put $m_y(0)=0$ and $m_y(1)=1$. Then $m_y$ is an extension of $m'_y$ to $I$. An easy verification shows that $m_y$ is increasing.  Now we prove that $m_y:I \longrightarrow I$ is strictly increasing.
Assume that $s,t\in (0,1)$ and $s<t$.  Being  $A_{\mu(y)}$ dense in $I$ there are functions $ f,g \in A$ such that  $ s< f(\mu(y)) <  g(\mu(y)) < t$. Since $ y \in Y_0 $ it follows from the definition of $m_y$ that

\[m_y(s) \leq m_y(f(\mu(y))) < m_y(g(\mu(y))) \leq m_y(t) \] that is $m_y(s)<m_y(t)$. Hence $m_y$ is  strictly increasing on $(0,1)$. It is obvious that $m_y(s)>0$ for all $s\in (0,1)$.  We note that $m_y(s)<1$ for all $s\in (0,1)$. Indeed, if $m_y(s)=1$ for some $s\in (0,1)$, then  choosing $f\in A$ with $s<f(\mu(y))<1$ we have $\varphi(f)(y)<1$, since $y\in Y_0$. On the other hand,
\[\varphi(f)(y)=m_y(f(\mu(y)))\ge m_y(s)=1,\]
a contradiction.  Hence $m_y:I \longrightarrow I$ is strictly increasing.

We shall show that $m_y$ is continuous at each point $s\in I$.  First assume that $s\in (0,1)$ is given. Let $ \{s_n\}$ and  $\{t_n\}$ be strictly increasing, respectively strictly decreasing sequences in $(0,1)$  converging to $s$.  Since $m_y$ is strictly increasing it suffices to show that
$\lim m_y(s_n) = \lim m_y(t_n)$.
Assume on the contrary that $\lim m_y(s_n) < \lim m_y(t_n)$. Since $B_y$ is dense in $I$, there exist
$g,h\in A$ such that:
\[ \lim m_y(s_n) < \varphi(g)(y) < \varphi(h)(y)
< \lim m_y(t_n). \]
Using the density of $A_{\mu(y)}$ we can choose a strictly increasing sequence $\{s'_n\}$ and a strictly decreasing sequence $\{t'_n\}$ in  $A_{\mu(y)}\cap (0,1)$ converging  to $s$ such that for all $k\in \BN$
\[ m_y( s'_k) < m_y( s_k) <  m_y( t_k) < m_y( t'_k).\]
Thus for each $k\in \BN$ we have
\[  m_y(s'_k) < \varphi(g)(y) < \varphi(h)(y)
< m_y(t'_k). \]
Since $y \in Y_0$ and $m_y$ is strictly increasing we conclude that for each $k\in \BN$
\[s'_k < g(\mu(y)) < h(\mu(y)) < t'_k,\]
which is impossible, since  $\lim s'_k =s=\lim t'_k $. Hence $m_y$ is continuous on $(0,1)$.   By a similar argument we can show that $m_y$ is continuous at points 0 and 1.

Finally,  using the mean value theorem, we deduce that $m_y$ is surjective.
 \end{proof}

The above lemma completes the proof of Theorem \ref{t2}. $\Box$

We note that, in Theorem \ref{t2} if $\varphi$ is assumed to be continuous,  then by Lemma \ref{p1} and Proposition \ref{p2}, $I_{\varphi^{-1}} = X$, that is $Y_0=I_\varphi$. Hence it follows from the theorem  that  there exist a continuous bijection $m_y:I \longrightarrow I$ such that
$\varphi(f)(y) = m_y(f(\mu(y)) )$ holds for all $f\in A$ and $y \in I_\varphi$. This gives the next corollary.

 \begin{cor}\label{c2}
Let $X$ and $Y$ be compact Hausdorff spaces, $A$ and $B$ be sublattices of $C(X,I)$ and $C(Y,I)$, respectively, which are closed under multiplication and  have Urysohn's
 property. Assume, furthermore, that the evaluations of $A$ and $B$ are dense in $I$. Then for any lattice isomorphism $\varphi:A \longrightarrow B $, the following statements are equivalent:

{\rm (i)} $\varphi$ is a homeomorphism.

{\rm (ii)} $\varphi$ is strictly increasing in both directions.

{\rm (iii)} There are a homeomorphism $ \mu : Y \longrightarrow X $
and a family $\{m_y\}_{y\in Y}$ of strictly increasing continuous
bijections on $I$ such that $\varphi(f)(y) = m_{y}( f(\mu(y)))$,
for all $y \in Y$ and all $f \in A$.
\end{cor}

Next theorem gives an explicit description of lattice isomorphisms between the lattices of interval-valued Lipschitz functions. In this
context the given lattice isomorphism is automatically continuous with respect to the uniform convergence topology.

\begin{theorem}
Let $(X,d)$ and $(Y,d^{\prime})$ be compact metric spaces and
$\varphi : \Lip(X,I) \longrightarrow \Lip(Y,I)$ be a lattice
isomorphism. Then there exist a homeomorphism $ \mu : Y
\longrightarrow X $ and a family $\{m_y\}_{y\in Y}$ of continuous
increasing bijections on $I$  such that $\varphi(f)(y) =
m_{y}(f(\mu(y)))$, for all $f \in \Lip(X,I)$ and $y\in Y$.
\end{theorem}
\begin{proof}
Since $\Lip(X,I)$ and $\Lip(Y,I)$ are sublattices of $C(X,I)$ and $C(Y,I)$, having Urysohn's property, it follows from Theorem \ref{t2} that $\varphi$ induces a homeomorphism  $\mu : Y \longrightarrow X$. As $\Lip(X,I)$ and $\Lip(Y,I)$ are closed under the multiplication and contain the constants, it suffices, by the second part of Theorem \ref{t2},  to show that $Y_0=Y$, that is  $I_{\varphi^{-1}} = X$ and $I_\varphi=Y$. We prove the first equality, since the other one is proven in a similar manner.
 Assume on the contrary that there exists a point $x_0 \in X$ such that $x_0\notin  I_{\varphi^{-1}}$. Let $y_0 \in Y$ such that $ \mu(y_0) = x_0$. Since $x_0 \notin I_{\varphi^{-1}}$,  there exist functions
 $f,g \in \Lip(X,I)$ with  $f(\mu(y_0)) = g(\mu(y_0))$ such that $ \varphi(f)(y_0) < \varphi(g)(y_0)$. We should note
 that $x_0$ is not an isolated  point since by Lemma  \ref{c1}, $I_{\varphi^{-1}}$
 contains all isolated points of $X$.  Using the density of $Y_0$, we can choose a sequence $\{x_n\}$ in $\mu(Y_0)$ converging to $x_0$ such that for each $i\in \BN$, $0<d(x_i,x_0)\le \frac{d(x_{i-1},x_0)}{2}$. We put  $\alpha_i=d(x_i,x_0)$ for all $i\in \BN$.
 We note that for $i,j\in
\BN$ with $i<j$ we have
\begin{equation}\label{1}
d(x_{i}, x_{j}) \geq d(x_i, x_0)-d(x_0, x_j)  \geq \frac{\alpha_{i}}{2}
\end{equation}
Let $ h: \lbrace x_{n} \rbrace \cup \lbrace x \rbrace
\longrightarrow I $ be defined by
\begin{equation}
h(x_{i})=\left\{
\begin{array}{lll}
f(x_{i}),&  \qquad i\; \text{is even}\\
g(x_{i}),& \qquad i\;  \text{is odd}\\
f(x_{i}),&  \qquad {i =0}    \\
\end{array}
\right .
\end{equation}
We show that $h$ is a Lipschitz function on the closed subset
$\{x_n\}\cup \{x_0\}$ of $X$. To do this assume that $M=\max(L(f),
L(g))$ where $L(f)$ and $L(g)$ are Lipschitz constants of $f$ and
$g$, respectively. Clearly for the given pair $i,j\in \BN$, if both $i,j$ are either even  or odd we have $|h(x_{i})- h(x_{j})| \leq M
d(x_{i},x_{j})$. Hence we assume that $i$ is even and $j$ is odd.
Without loss of generality we assume that $i < j$. Then we have
$|h(x_{i}) - h(x_{j})|\leq L(f) d( x_{i}, x_{0} )  + L(g) d(x_{j}
, x_{0}) \leq M\,(\alpha_{i} + \alpha_{j}) \leq 2M \alpha_{i}$. By
(\ref{1}) we have $|h(x_{i}) - h(x_{j}) | \leq 2M \alpha_{i}
  \leq 4M d( x_{i} , x_{j} )$. This shows that $h$ is a Lipschitz function on $\{x_n\}\cup \{x_0\}$. Hence, by (\cite{Wea}, Theorem 1.5.6(a)),  we can extend $h$ to a Lipschitz function $H: X \longrightarrow I$ with the same  Lipschitz constant $L(h)$.
For each $i\in \BN$ we put $y_i=\mu^{-1}(x_i)$. Since  $\{y_i\}$ is a sequence in $Y_0$ and $H(\mu(y_i)=f(\mu(y_i))$ whenever $i$ is even, we conclude that
$\varphi(H)(y_i) = \varphi(f)(y_i)$ for even $i\in \BN$. Hence tending $i\to \infty$ we get $\varphi(H)(y_0)=\varphi(f)(y_0)$. Similarly we have
$\varphi(H)(y_j) = \varphi(f)(y_j)$ for all odd $j\in \BN$ and consequently $\varphi(H)(y_0)=\varphi(g)(y_0)$ which is impossible, since $\varphi(f)(y_0)\neq \varphi(g)(y_0)$.
\end{proof}


\begin{thebibliography}{9}

\bibitem{Ar}
J. Araujo, {\em Multiplicative bijections of semigroups of
interval-valued continuous functions}, Proc. Am. Math. Soc. 137
(2009) 171--178 .


\bibitem{San1}
Cabello S\'anchez F.: Homomorphisms on lattices of continuous
functions. Positivity 12,  341--362  (2008)

\bibitem{San2}
 Cabello S\'anchez F.,  Cabello Sanchez J.,  Ercan Z.,  \"Onal S.:  Memorandum on multiplicative bijections and order. Semigroup Forum 79,  193-209 (2009)



\bibitem{Er-On}
Ercan Z., \"Onal S.:  An answer to a conjecture on
multiplicative maps on C(X,I). Taiwanese J. Math. 12,  
537--538 (2008)

\bibitem{Fel}
Feldman W.A.:  Lattice preserving maps on lattices of
continuous functions. J. Math. Anal. Appl. 404,  310-316  (2013)



\bibitem{Kap1} Kaplansky I.: Lattices of continuous functions. Bull. Am. Math. Soc. 53,  617--623 (1947)


\bibitem{Kap2}  Kaplansky I.:  Lattices of continuous functions II. Am. J.
Math. 70, 626--634 (1948)

\bibitem{Le-Li}
Leung D.H., Li L.:  Order isomorphisms on function
spaces, Studia Math. 219, 123-138 (2013)

\bibitem{Mar0} Marovt J.: Order preserving bijections of $C_+(X)$. Taiwanese J. Math.  14, 667--673 (2010)

\bibitem{Mar2}
Marovt J.:  Multiplicative bijections of $C(X,I)$. Proc.
Amer. Math. Soc. 134,  1065--1075 (2005)

\bibitem{Mar1}
Marovt J.:  Order preserving bijections of $ C(X, I) $. J.
Math. Anal. Appl.  311,  567--581  (2005)

\bibitem{Mol}
Moln\'ar L.:  Sequential isomorphisms between the sets of von
Neumann algebra effects. Acta Sci. Math. (Szeged) 69,
755-772 (2003)

\bibitem{Wea}  Weaver N.:  Lipschitz algebras,  Second edition,  World Scientific, Singapore (1999)

\end{thebibliography}
\end{document}